\documentclass[10pt,a4paper]{article}

\usepackage[utf8]{inputenc}
\usepackage[T2A]{fontenc}
\usepackage[english]{babel}

\usepackage{amsmath,amsfonts,amssymb,amsthm}
\usepackage{fullpage}
\usepackage{url}

\newtheorem{theorem}{Theorem}
\newtheorem{lemma}[theorem]{Lemma}
\newtheorem{corollary}[theorem]{Corollary}

\newcommand{\lcm}{\mathrm{lcm}}

\title{On $p$-adic approximation of sums of binomial coefficients}
\author{}
\date{}
\begin{document}
	\maketitle
	\begin{flushright}
		{\sc\bf Rustem R. Aidagulov} \\ {\it Moscow State University, Russia}\\
		{\sc\bf Max A. Alekseyev} \\ {\it George Washington University, USA}		
	\end{flushright}

UDC 511.172


\begin{abstract}
We propose higher-order generalizations of Jacobsthal's $p$-adic approximation for binomial coefficients.
Our results imply explicit formulae for linear combinations of binomial coefficients $\binom{ip}{p}$ ($i=1,2,\dots$) that are divisible 
by arbitrarily large powers of prime $p$.
\end{abstract}

\section{Introduction}

Finding a power of prime $p$ dividing a given integer can be viewed as establishing its $p$-adic precision.
Namely, the power of $p$ dividing the integer shows how small this integer is the $p$-adic metric, since the divisibility $p^r\mid n$ implies 
the approximation (congruence) $n=O(p^r)$.

The problem of finding $p$-adic distance between the binomial coefficients $\binom{ap}{bp}$ and $\binom{a}{b}$
is attributed to Lucas~\cite{Gashkov2000}.
In 1878, Lucas proved~\cite{Granville1997,Mestrovic2014} that
$$\binom{a}{b} = \prod_{i=0}^d \binom{a_i}{b_i} + O(p),$$
where $a=a_0+a_1p+\dots+a_dp^d$ and $b=b_0+b_1p+\dots+b_dp^d$ are the base $p$ representations of integers $a$ and $b$. 
Even earlier, in 1869, Anton obtained a stronger result:
$$\frac{(-1)^\ell}{p^\ell}\binom{a}{b} = \prod_{i=0}^d\frac{a_i!}{b_i!c_i!} + O(p),$$
where $c_i$ are the base $p$ digits of the difference $c=a-b=c_0+c_1p+\dots+c_dp^d$, and $\ell=\nu_p\left(\tbinom{a}{b}\right)$.
Here $\nu_p(x)$ denotes the $p$-adic valuation of $x$ (i.e., the largest power of $p$ dividing $x$).
In 1852, Kummer showed that $\ell$ equals the number of carries in the addition of integers $a-b$ and $b$ in base $p$ arithmetic.

It is easy to see that the result of Anton implies the following approximation:
\begin{equation}\label{Eq0}
\binom{ap}{bp}/\binom{a}{b} = 1 + O(p),
\end{equation}
which was already known to Kummer. 
For the sake of convenience, we consider \emph{modified} ($p$-adic) factorials and binomial coefficients defined by the formulae:
$$a!_p=\prod_{k=1\atop p\nmid k}^a k,\qquad  \binom{a}{b}_p=\frac{a!_p}{b!_p(a-b)!_p}.$$
It is easy to see that modified binomial coefficients are integer.
Theorems that extend the approximation~\eqref{Eq0} to higher powers of $p$ of the form
\begin{equation}\label{Eq1}
\binom{ap}{bp}/\binom{a}{b} = \binom{ap}{bp}_p = 1 + O(p^r),\quad r>1
\end{equation}
are referred to as those of \emph{Wolstenholme type}~\cite{Granville1997,Mestrovic2014}.
The first theorem of this type proved by Babbage in 1819 states that
\begin{equation}\label{Eq2}
\binom{2p}{p}_p=\frac{1}{2}\binom{2p}{p}=\binom{2p-1}{p-1} = 1 + O(p^r)
\end{equation}
for $r=2$ and all primes $p>2$. This approximation corresponds to \eqref{Eq1} for $a=2$ and $b=1$. 
Wolstenholme extended the Babbage result by proving \eqref{Eq2} for $r=3$ and primes $p>3$;
he also posed the problem of finding primes $p$ for which \eqref{Eq2} holds with $r=4$.
Such primes are now named after him (as of 2016, only two Wolstenholme primes are known: $16843$ and $2124679$).
For any integers $a>b>0$, $r=3$, and prime $p>3$, the approximation \eqref{Eq1} was proved by Ljunggren in 1949~\cite{Mestrovic2014}, 
which was extended by Jacobsthal in 1952 to
\begin{equation}\label{Eq3}
\binom{ap}{bp}_p = 1 + O(p^r),\quad r=3+\nu_p(ab(a-b)),\quad p>3.
\end{equation}
Moreover, $r$ here can be further increased by 1 if prime $p$ divides Bernoulli number $B_{p-3}$.
Nowadays, partial cases of all these results are often offered as problems in mathematical contests 
and journals for school students~\cite{Fuks1970,Vinberg2008}.

Recently the second author proposed the following generalization of the Wolstenholme congruence to arbitrarily large powers of primes:
\begin{theorem}\label{Th0}
For any integers $n,m\geq 1$ and any prime $p>2n+1$, the linear combination of modified binomial coefficients
\begin{equation}\label{Eq4}
\sum_{j=0}^{n} (-1)^j\binom{2n+1}{j}\frac{2(n-j)+1}{2n+1}\cdot \binom{(n+1-j)m}{m}_p
\end{equation}
is divisible by $p^{(2n+1)\nu_p(m)}$.
\end{theorem}

We remark that the coefficients of the modified binomial coefficients in \eqref{Eq4} are integer since 
$\binom{2n+1}{j}\frac{2(n-j)+1}{2n+1}=\binom{2n-1}{j}-\binom{2n-1}{j-2}=\binom{2n}{j}-\binom{2n}{j-1}$.

Theorem~\ref{Th0} implies a similar statement for conventional binomial coefficients:
\begin{corollary}\label{Col0}
For any integer $n\geq 1$ and any prime $p>2n+1$, the linear combination of binomial coefficients
$\sum_{i=1}^{n+1} c_i \binom{ip}{p}$ is divisible by $p^{2n+1}$, where the coefficients
$$c_i = (-1)^{i-1}(2n+1)\frac{\mathrm{lcm}(1,2,\dots,2n)\binom{2n+1}{n+1-i}(2i-1)}{\binom{2n+1}{n}i}$$ 
are integer and setwise coprime.
\end{corollary}

For example, for $n=1,2,3$, Corollary~\ref{Col0} gives the following divisibility by powers of primes:
$$
\begin{array}{ll}
p^3 \mid 2  - \binom{2p}{p}, & p>3;\\
p^5 \mid 12 - 9\binom{2p}{p} + 2\binom{3p}{p}, &  p>5;\\
p^7 \mid 60 - 54\binom{2p}{p} + 20\binom{3p}{p} - 3\binom{4p}{p}, & p>7.
\end{array}
$$
Here, the divisibility for $n=1$ is equivalent to the Wolstenholme congruence \eqref{Eq2} for $r=3$.
The coefficients $c_i$ form the sequence \texttt{A268512}, while the quotients for $n=1,2,3$ are given by the sequences \texttt{A087754}, \texttt{A268589}, and \texttt{A268590} in the OEIS~\cite{OEIS}.

In the present work, we prove the following theorem, which implies Theorem~\ref{Th0} as a particular case.

\begin{theorem}\label{Th1}
Let $q$ be a power of a prime $p$, $b>0$ be an integer, and $a=a_0$, $a_1$, $\dots$, $a_n$ be distinct integers not smaller than $b$. 
Then there exists a unique set of rational numbers
$$y_i=\prod_{k=1\atop k\neq i}^n \frac{(a-a_k)(a+a_k-b)}{(a_i-a_k)(a_i+a_k-b)}$$
that provides most accurate approximation for the modified binomial coefficient $\binom{aq}{bq}_p$ additively
$$\binom{aq}{bq}_p=\sum_{i=1}^n y_i\binom{a_iq}{bq}_p +O(p^r),$$
as well as multiplicatively\footnote{Note that \eqref{Eq0} implies $\binom{a_iq}{bq}_p=1+O(p)$, thus
taking $\binom{a_iq}{bq}_p$ to a rational power is well-defined via the binomial expansion:
		$(1+t)^y=1+\binom{y}{1}t+\binom{y}{2}t^2+\dots$.}
	$$\binom{aq}{bq}_p=\prod_{i=1}^n \left(\binom{a_iq}{bq}_p\right)^{y_i} (1+ O(p^r)).$$
Moreover, for any prime $p>\max\{2n+1,\,a_i+a_k-b\ :\ 1\le i < k \le n\}$, the order of approximation is at least\footnote{The term $\epsilon$ 
in the formula for $r$ here is similar to the condition $p\mid B_{p-3}$ increasing the approximation order 
in the Jacobsthal congruence \eqref{Eq3}.}
	$$r=(2n+1)\nu_p(q)+\nu_p(g_0(a))+\nu_p(b)+\epsilon,$$
where $g_0(x)=\prod_{k=1}^n (x-a_k)(x+a_k-b)$, $\epsilon = \min\{t,\nu_p(B_{M-2n})\}$, $t=\nu_p(bq)$, and $M=p^{t-1}(p-1)$.
\end{theorem}

Theorem~\ref{Th1} also generalizes the Jacobsthal congruence, which is obtained here when $n=1$, $a_1=b$, and $q=p$.

\section{Proof of Theorem~\ref{Th1}}

Suppose that the conditions of Theorem~\ref{Th1} hold.
Our first goal is finding rational numbers $y_1,\dots,y_n$ not depending on $p$
that approximate the modified binomial coefficient $\binom{aq}{bq}_p$ in $p$-adic metric as
\begin{equation}\label{Eq9}
\binom{aq}{bq}_p-\sum_{i=1}^ny_i\binom{a_iq}{bq}_p=O(p^r)
\end{equation}
with the largest possible $r$. 
We will see below that there exists a unique set of such rational numbers.
Uniqueness here follows from the fact that $y_i$ do not depend on prime $p$, i.e., the approximation is the best possible for all large enough $p$.
We will need the following lemma.

\begin{lemma}\label{Lem1}
	Let $S = \left\{ \frac{bq}{2}-\ell\ :\ 0\le \ell <\frac{bq}{2},\ p\nmid \frac{bq}{2}-\ell\ \right\}$.
	Denote $N=|S|=\frac{bq(p-1)}{2p}$.\footnote{When $p=q=2$ and $b$ is odd, we assume that $S$ is a multiset where the element $\ell=0$
comes with multiplicity $\frac 12$, and thus $N$ is a half-integer.}
Then the modified binomial coefficients $\binom{a_iq}{bq}_p$ can be expressed in the form:
	\begin{equation}\label{Eqff}
	\binom{a_iq}{bq}_p = \frac{f(z_i)}{f\left(\frac{b^2}{4}\right)},
	\end{equation}
where $z_i=\left(a_i-\frac b2\right)^2$ and
	\begin{equation}\label{Eq10f}
	f(x)=\prod_{k\in S} \left(1-x\frac{q^2}{k^2}\right) = \sum_{i=0}^{N}(-xq^2)^i\sigma_i.
	\end{equation}
Here $\sigma_i$ are elementary symmetric polynomials of numbers $\frac{1}{k^2}$, $k\in S$.
\end{lemma}

\begin{proof} We have
	$$
	\binom{a_iq}{bq}_p
	=\prod_{k\in S}\frac{((a_i-b/2)q+k)((a_i-b/2)q-k)}{(bq/2+k)(bq/2-k)}
	=\prod_{k\in S}\frac{1-x(\frac{q}{k})^2}{1-(\frac{bq}{2k})^2}
	=\frac{f(z_i)}{f(\frac{b^2}{4})}. 
	$$
\end{proof}

By Lemma~\ref{Lem1}, after multiplication of the left-hand side of \eqref{Eq9} by $f\left(\frac{b^2}{4}\right) =\pm 1 + O(q)$, it takes the form:
\begin{equation}\label{Eq9dif}
f(z_0) -\sum_{i=1}^n y_i f(z_i).
\end{equation}
Hence, we need to find rational numbers $y_i$ giving the best (in the $p$-adic metric) approximation for the value $f(z_0)$ from the values $f(z_i)$. 
This can be achieved by choosing $y_i$ in such a way that in the difference \eqref{Eq9dif} all small powers of $p$ disappear, 
which by \eqref{Eq10f} corresponds to solving the following system of linear equations:
\begin{equation}\label{Eq11}
z_0^d=\sum_{i=1}^n y_iz_i^d,\qquad d=0,1,\dots,n-1.
\end{equation}
Since all $a_i\ge b$ and pairwise distinct, we have $z_i\neq z_j$ for all $i\neq j$. This implies that the determinant of the system \eqref{Eq11} 
representing a Vandermonde determinant is nonzero. 
To solve the system \eqref{Eq11}, we notice that it implies that for any polynomial $g(z)$ of degree smaller than $n$, we have 
$g(z_0)=\sum_{i=1}^n y_ig(z_i).$
Taking consecutively polynomials $g_i(z)=\prod_{j\neq i}(z-z_j)$, we obtain $g_i(z_0)=y_ig_i(z_i)$. 
Therefore, the values of $y_i$ are uniquely determined as
\begin{equation}\label{Eq12}
y_i=\prod_{k=1\atop k\neq i}^n\frac{z_0-z_k}{z_i-z_k}=\prod_{k=1\atop k\neq i}^n \frac{(a-a_k)(a+a_k-b)}{(a_i-a_k)(a_i+a_k-b)}.
\end{equation} 
These rational numbers $y_i$ are $p$-adic integers, since $p>\max_{1\leq i<k\leq n} a_i+a_k-b$. 

It can be easily seen that the maximization of $r$ in the multiplicative approximation: 
\begin{equation}\label{Eq13}
\binom{aq}{bq}_p \cdot \prod_{i=1}^n\left(\binom{a_iq}{bq}_p\right)^{-y_i}=1+O(p^r),
\end{equation}
results in the same equations \eqref{Eq11} and solutions \eqref{Eq12}.\footnote{We remark that
the equation \eqref{Eq11} for $d=0$ here is necessary to cancel factors $f\left(\tfrac{b^2}{4}\right)$ after 
substitution of expressions \eqref{Eqff} into the left-hand side of \eqref{Eq13}.}

Now let us find the order of approximation \eqref{Eq9}, i.e., estimate
$$\sum_{i=0}^N (-q^2)^i\sigma_i\Delta_i,\qquad \Delta_i=z_0^i-y_1z_1^i-\dots-y_nz_n^i.$$
For $i<n$, we have $\Delta_i=0$. 
For $i=n$, we have $z^n=g(z)+r_n(z)$, $\deg(r_n(z))<n$, where 
$$g(z)=(z-z_1)(z-z_2)\cdots (z-z_n).$$ 
It follows that $\Delta_n=g(z_0)$. 
From the representation $z^{n+1}=(z+z_1+z_2+\dots+z_n)g(z)+r_{n+1}(z)$, $\deg(r_{n+1}(z))<n$, we obtain 
$\Delta_{n+1}=(z_0+z_1+\dots+z_n)g(z_0)$. 
Similarly, for $i>n$, we have $\Delta_{i}=f_i(z_0,z_1,\dots,z_n)g(z_0)$.
One can represent $f(x)=f_1(x)g(x)+r(x)$, $\deg(r(x))<n$, and estimate the error term as $f_1(z_0)g(z_0)$. 
However, we need an expansion over growing powers of $q$ and for this purpose will use the following formula for the remainder:
$$
\sum_{i=n}^N(-q^2)^i\sigma_i\Delta_i=g(z_0)(-1)^nq^{2n}\left(\sigma_n-q^2\sigma_{n+1}(z_0+z_1+\dots+z_n)+O(q^4\sigma_{n+2})\right).
$$
Therefore, the order of approximation \eqref{Eq9} for primes $p>n$ is given by the formula:
\begin{equation}\label{EqR}
r=(2n+1)\nu_p(q)+\nu_p(g(z_0))+\nu_p\left(\frac{\sigma_n}{q}\right).
\end{equation}
The order of multiplicative approximation is the same. 

Notice that formula \eqref{EqR} for the error term of approximation generalizes the Jacobsthal formula \eqref{Eq3}.	
Indeed, $(2n+1)\nu_p(q)$ in \eqref{EqR} corresponds to the term $3$ in \eqref{Eq3} ($n=1$), the next term $\nu_p(g(z_0))$ 
corresponds to $\nu_p(a-b)+\nu_p(a)$, and $\nu_p\left(\frac{\sigma_n}{q}\right)$ represents an analog of $\nu_p(b)+\nu_p(B_{p-3})$.
To prove the last claim, let us estimate $\sigma_n$, using the Newton--Girard formulae:
\begin{equation}\label{Eq14}
n\sigma_n=\sum_{i=1}^n(-1)^{i-1}\sigma_{n-i}s_i,
\end{equation}
where $s_i$ denotes the corresponding power sums:
$$
s_i=\sum_{k\in S}k^{-2i}\equiv \sum_{k\in S}k^{M-2i} \pmod{p^{2t}},
$$
where $t=\nu_p(bq)$ and $M=p^{t-1}(p-1)$.
We will show that for primes $p>2n+1$ and $i=1,2,\dots,n$,
\begin{equation}\label{Eq15}
\nu_p(s_i)\ge \nu_p(bq)+\theta_i, \quad\text{where}\quad \theta_i=\min\{t,\nu_p(B_{M-2i})\}.
\end{equation}

Let 
\begin{equation}\label{EqSi}
S_i=\sum_{\ell=1\atop p\nmid\ell}^{bq-1} \ell^{-2i}=\sum_{k\in S}\left(\frac{bq}{2}-k\right)^{-2i}+\left(\frac{bq}{2}+k\right)^{-2i}.
\end{equation}
To avoid negative powers in the last formula, we replace the negative degrees $-2i$ with $m=M-2i$. 
This gives us the following estimate:
$$S_i=\frac{1}{m+1}\sum_{k=1}^{m+1}B_{m+1-k}(bq)^k+O(p^r)=bqB_m+\frac{m(m-1)}{6}B_{m-2}(bq)^3+O((bq)^4),$$
where we took into account the evenness of $2i$ and $m$. 
Hence, $\nu_p(S_i)\ge \nu_p(bq)+\theta_i$. On the other hands, expressing
the terms of $S_i$ in \eqref{EqSi} via $s_i$ (again replacing $-2i$ with $m$ and using the binomial expansion), 
we get $S_i=2s_i+O\left((bq)^2\right)\equiv 2s_i\pmod{p^{2t}}$. 
This implies the required estimate \eqref{Eq15} for $s_i$. 

From formula \eqref{Eq14}, we get the following expression for $\sigma_n$:
$$
\begin{array}{lll}
\sigma_n & = & \frac{(-1)^{n-1}}{n}\left(s_n-\sigma_1s_{n-1}+\dots+(-1)^{n-1}\sigma_{n-1}s_1\right)\\
& = & \frac{(-1)^{n-1}}{n}\left(s_n-\sum_{j=1}^{n-1}\frac{s_{n-j}}{j}\left(s_j+\sum_{i=1}^{j-1}(-1)^{i-1}s_{j-i}\sigma_i\right)\right)\\
& = & \dots
\end{array}
$$
Eventually this leads us to the formula:
\begin{equation}\label{Eq16}
\sigma_n=(-1)^n\sum_{k=1}^n(-1)^k\sum_{j_1+j_2+\dots+j_k=n}\frac{s_{j_1}s_{j_2}\cdots s_{j_k}}{j_1(j_1+j_2)\cdots (j_1+j_2+\dots+j_k)}.
\end{equation}
We remark that primes greater than $n$ do not divide the denominators of terms in \eqref{Eq16}. 
For a prime $p>2n+1$, estimate \eqref{Eq15} implies that the sum of terms in \eqref{Eq16} with a fixed $k$
can be estimated as $O(p^{k\cdot \nu_p(bq)})$. Hence, from \eqref{Eq16} it follows that $\nu_p(\sigma_n) \ge \min\{\nu_p(s_n),2\nu_p(bq)\}$. 
From estimate \eqref{Eq15} for $s_n$, we further get that
$$\nu_p(\sigma_n)\ge \min\{\nu_p(bq)+\theta_n,2\nu_p(bq)\} = \nu_p(bq)+\theta_n.$$
From \eqref{EqR} it now follows that for a prime $p>\max\{2n+1,\,a_i+a_k-b\,:\,0\le i<k\le n\}$,
the order of approximation \eqref{Eq9} is at least
$$r=(2n+1)\nu_p(q)+\nu_p(g(z_0))+\nu_p(b)+\theta_n.$$
This completes the proof of Theorem~\ref{Th1}.

\section{Proof of Theorem~\ref{Th0}}

Theorem~\ref{Th0} easily follows from Theorem~\ref{Th1} as a particular case
with $a_i=ib$ and $m=bq$.
Theorem~\ref{Th0} can also be proved directly, using the forward difference operator $\Delta f(x)=f(x+1)-f(x)$.
Clearly, $\Delta$ decreases the degree of a polynomial by 1 (as the conventional differentiation), and sends constants to 0. 
Correspondingly, its $m$-th power of $\Delta$:
$$\Delta^m f(x) = \sum_{i=0}^m (-1)^{m-i}\binom{m}{i} f(x+i)$$ 
decreases the degree of a polynomial by $m$.

\begin{proof}[Proof of Theorem~\ref{Th0}]
Similarly to Lemma~\ref{Lem1}, we can represent the modified binomial coefficient $\binom{xm}{m}_p$ for a fixed $m$ as a polynomial of $x$:
$$\binom{xm}{m}_p=f(x)=\prod_{k=1\atop p\nmid k}^{m} (1-\frac{xm}{k})=f(1-x)=\sum_{i=0}^N (-1)^i\sigma_i(xm)^i.$$
Then the sum \eqref{Eq4} can be stated in the form:
\[
\begin{split}
S & = \sum_{j=0}^n (-1)^j\binom{2n+1}{j}\frac{2n+1-2j}{2n+1}f(n+1-j)\\
& =\sum_{j=0}^n(-1)^j \binom{2n}{j}f(n+1-j)\ +\ \sum_{j=0}^n(-1)^{j+1} \binom{2n}{j-1} f(n+1-j).
\end{split}
\]
Our goal is to represent $S$ via operator $\Delta$, using the identity $f(x)=f(1-x)$.
Let us rewrite the parts of $S$ as follows:
\[
\begin{split}
\sum_{j=0}^n(-1)^j\binom{2n}{j}f(n+1-j) &= \sum_{i=-n}^0(-1)^{n+i}\binom{2n}{n+i}f(i),\\
\sum_{j=0}^n(-1)^{j+1}\binom{2n}{j-1}f(n+1-j) &=\sum_{i=1}^{n+1}(-1)^{n+i}\binom{2n}{n+i}f(i).
\end{split}
\]
Hence, we have
\[
\begin{split}
S & =\sum_{i=-n}^{n}(-1)^{n+i}\binom{2n}{n+i}f(i)=\frac 12\sum_{i=-n}^{n}(-1)^{n+i}\binom{2n}{n+i}(f(i)+f(1-i))\\
 & =\frac 12 \sum_{i=-n}^n(-1)^{n+i}\binom{2n}{n+i}(f(i)+f(i+1))=\left.\Delta^{2n}(f(x)-f(-x))\right|_{x=-n}.
\end{split}
\]
Since the function $f(x)-f(-x)$ is odd, the operator $\Delta^{2n}$ eliminates all powers of $x$ below $2n+1$, implying that $S$ is divisible by $p^{(2n+1)\nu_p(b)}$.
\end{proof}

Now we prove Corollary~\ref{Col0}. 
\begin{proof}[Proof of Corollary~\ref{Col0}]
Theorem~\ref{Th0} for $m=p$ implies that $\sum_{i=1}^{n+1} (-1)^{i-1}\cdot t_i\cdot \binom{ip}{p}$ is divisible by $p^{2n+1}$ for any prime $p>2n+1$, 
where the coefficients $t_i=\binom{2n+1}{n+1-i}\frac{2i-1}{(2n+1)i}$ may be not integer. In particular, $t_{n+1}=\frac{1}{n+1}$.
Notice that for $i\le n$, we also have $t_i=\binom{2n}{n-i}\frac{2i-1}{(n+1-i)i}$.
For any prime $r$, let
\begin{equation}\label{eq:ellr}
\ell_r=\max_{1\leq i\leq n+1} -\nu_r(t_i).
\end{equation}
Since $i\cdot t_i$ (which represent the coefficients in \eqref{Eq4} up to signs) are integer, for $r>n+1$ we have $\ell_r =0$, 
while for $r\le n+1$ we have $\ell_r\ge -\nu_r(t_{n+1})\ge 0$. 
To turn the coefficients $t_i$ into integers, they need to be multiplied by a positive integer $L=\prod_{r\le n+1} r^{\ell_r}$. Moreover, $L$ is the minimum such number
and the integer coefficients $c_i=t_i\cdot L$ are setwise coprime. 
Hence, our goal is to find an explicit formula for $L$, which is equivalent to finding the value of $\ell_r$ for all prime $r\leq n+1$.

Let $r\le n+1$ be a prime. For each $i=1,2,\dots,n$, we have
\begin{equation}\label{eq:lpi}
-\nu_r(t_i)=\nu_r(i)-\nu_r\left(\binom{2n+1}{n+1-i}\right)-\nu_r\left(\frac{2i-1}{2n+1}\right),
\end{equation}
while $-\nu_r(t_{n+1})=\nu_r(n+1)$. 

Let $i=i_0+i_1r+\dots+i_kr^k$ and $n+1=n_0+n_1r+\dots+n_kr^k$ be the base $r$ representations of $i$ and $n+1$, where integer $k\geq 1$ satisfies $r^k\le n+1<r^{k+1}$. It is clear that $-\nu_r(t_i)\le -\nu_r(t_{j_l})$, where $l=\nu_r(i)$, $j_s=n+1-((n+1-i)\bmod r^{s+1})$. 
Upon replacement of $i$ with $j_l$, the first and third terms in \eqref{eq:lpi} do not change, while the second term may only increase. 
Hence, for maximization in \eqref{eq:ellr} it is enough to consider only the cases, when the base-$r$ digits of $i$ and $n+1$ satisfy the equalities: $i_s=n_s$ for $s\ge l$ and $i_s=0$ for $s<l$. 

If $\nu_r(i)=0$, then from \eqref{eq:lpi} it follows that
$$-\nu_r(t_i)=\nu_r(n+1)-\nu_r\left(\binom{2n}{n-i}\right)-\nu_r(2i-1)\le -\nu_r(t_{n+1})\le \ell_r.$$ 
If addition of $n+1$ and $n$ in base $r$ does not have a carry in the $l$-th (least significant) position, then 
$-\nu_r(t_{j_{l+1}})\ge -\nu_r(t_{j_l})+1$ since $\nu_r(j_{l+1})=\nu_r(j_l)$ increases. 
If a carry in the $l$-th position happens, it may follow by more carries, i.e., $-\nu_r(t_{j_l})\le -\nu_r(t_{j_{l+m}})$, where $l+m$ is the first position after $l$ with no carry. More precisely, for $s=0,1,\dots,m$, we have 
$-\nu_r(t_{j_{l+s}})=-\nu_r(t_{j_{l+m}})-(m-s)=-\nu_r(t_{j_l})+1-(m-s)$ under the condition that the corresponding base $r$ digits of $n+1$ are nonzero. Hence, the maximum of $-\nu_r(t_i)$ is achieved at $i=j_q$, where $q$ is the largest position with no carry when $n+1$ and $n$ are added in base $r$. It follows that
$$L=\prod_{r\le n+1} r^{\ell_r} = \frac{\lcm(1,2,\dots,2n)\cdot (2n+1)}{\binom{2n+1}{n}}.$$

Since for each prime $r\le n+1$, there exists an index $i$ such that $r\nmid c_i$, the coefficients $c_i=L\cdot t_i$ are integer and setwise coprime.
\end{proof}

\section{Concluding Remarks}

Theorem~\ref{Th1} covers that case of sums of binomial coefficients with upper indices being arbitrary multiples of $p$, 
but with a fixed lower index. 
Our analysis shows that generalizations of Theorem~\ref{Th1} to the case of arbitrary lower indices does not always lead to soluble
linear equations for the coefficients $y_i$, and even if solutions exist they can hardly be expressed explicitly.

We remark that there also exists a generalization of the Jacobsthal congruence to the case of composite modulus proposed by the first author.
Namely, the Jacobsthaln congruence can be expressed as 
\begin{equation}\label{eq:divm}
m^3\mid 6\cdot\sum_{d\mid m}\mu\left(\frac md\right)\binom{ad}{bd},
\end{equation}
where $m=p$ is prime and $\mu(\cdot)$ is the M\"obius function.
It turns out that congruence \eqref{eq:divm} holds also for an arbitrary positive integer $m$. 
This statement follows from the Jacobsthal congruence by considering the right-hand side of \eqref{eq:divm} modulo $p^{3\nu_p(n)}$ for every prime $p\mid m$. 
From \eqref{eq:divm} one can easily obtain a similar congruence:
\begin{equation}\label{eq:divmsign}
m^3\mid 12\cdot\sum_{d\mid m}(-1)^{m+d}\mu\left(\frac md\right)\binom{ad}{bd}.
\end{equation}
We remark that the factor 6 in \eqref{eq:divm} can be replaced with $M(a,b)=\frac{12}{\gcd(12,ab(a-b))}$ (it is easy to see that $M(a,b)\mid 6$), 
while for some $a,b$, the factor can be further decreased down to $\frac{1}{2}M(a,b)$.
Similarly, the factor 12 in \eqref{eq:divmsign} can be replaced with
$$
M'(a,b) = \frac{3}{\gcd(3,ab(a-b))}\cdot 2^\delta,\quad\text{where}\quad 
\delta = \begin{cases}
\min\{1,\nu_2(b)\}, & \text{if}\ \nu_2(a-b)=\nu_2(b),\\
2, & \text{otherwise},
\end{cases}
$$
while for some $a,b$ it can be further decreased down to $\frac{1}{2}M'(a,b)$. 
For example, for $(a,b)=(2,1)$, the quotients corresponding to factors $M(2,1)=6$ and $M'(2,1)=3$ are given by the sequences \texttt{A268592} and \texttt{A254593} in the OEIS~\cite{OEIS}.
Theorem~\ref{Th1} allows one to further generalize congruences \eqref{eq:divm} and \eqref{eq:divmsign} to higher powers of $m$.

\bibliographystyle{plain}
\bibliography{luc_english.bib}

\end{document}